\documentclass{amsart}

\usepackage{amssymb}

\usepackage[pdftex]{hyperref} 
\usepackage[capitalize]{cleveref}

\newcommand{\pref}[1]{(\ref{#1})}
\newcommand{\fullcref}[2]{\cref{#1}\pref{#1-#2}}
\newcommand{\csee}[1]{(see \cref{#1})}

\newcommand{\iso}{\cong}
\newcommand{\normal}{\triangleleft}

\newcommand{\integer}{\mathbb{Z}}

\renewcommand{\pmod}[1]{\ (\mathop{\mathrm{mod}}{#1})}

\DeclareMathOperator{\Cay}{Cay}
\DeclareMathOperator{\GL}{GL}

\DeclareMathOperator{\GF}{GF}
\DeclareMathOperator{\Aut}{Aut}

\newcommand{\qt}[1]{\overline{#1}}
\let\qtit=\qt
\newcommand{\bya}[1]{&\stackrel{\textstyle a}{\longrightarrow}&\qtit{#1}}
\newcommand{\byai}[1]{&\stackrel{\textstyle a^{-1}}{\longrightarrow}&\qtit{#1}}
\newcommand{\byb}[1]{&\stackrel{\textstyle b}{\longrightarrow}&\qtit{#1}}
\newcommand{\bybi}[1]{&\stackrel{\textstyle b^{-1}}{\longrightarrow}&\qtit{#1}}

\newcommand{\ul}[1]{\underline{#1}}

\makeatletter 
\swapnumbers
\def\swappedhead#1#2#3{%
  {\normalfont(\thmnumber{#2})}% I like parentheses around the theorem number
  \thmname{\@ifnotempty{#2}{~}#1}%
  \thmnote{ {\the\thm@notefont#3}}} % I removed the parentheses around the note
\def\thmhead@plain#1#2#3{%
  \thmname{#1}\thmnumber{\@ifnotempty{#1}{ }\@upn{#2}}%
  \thmnote{ {\the\thm@notefont#3}}} % I removed the parentheses around the note
\let\thmhead\thmhead@plain
\makeatother

\crefformat{case}{Case~#2#1#3}

 \newcounter{case}
 \newenvironment{case}[1][\unskip]{\refstepcounter{case}\bf
 \medskip \noindent Case \thecase\ #1. \it}{\unskip\upshape}
 \renewcommand{\thecase}{\arabic{case}}
\crefformat{case}{Case~#2#1#3}

 \newcounter{subcase}
 
\numberwithin{subcase}{case}
\crefformat{subcase}{Subcase~#2#1#3}

 \newcounter{subsubcase}
 
\numberwithin{subsubcase}{subcase}
\crefformat{subsubcase}{Subsubcase~#2#1#3}

 \newcounter{subsubsubcase}
 
\numberwithin{subsubsubcase}{subsubcase}
\crefformat{subsubsubcase}{Subsubsubcase~#2#1#3}

\numberwithin{equation}{section}

\theoremstyle{plain}
\newtheorem{thm}[equation]{Theorem}
\newtheorem*{thmstar}{Theorem}
\newtheorem{prop}[equation]{Proposition}
\newtheorem{cor}[equation]{Corollary}

\newtheorem{assump}[equation]{Assumption}
\newtheorem{lem}[equation]{Lemma}
\newtheorem{FGL}[equation]{Lemma}

\crefformat{thm}{Theorem~#2#1#3}
\crefformat{prop}{Proposition~#2#1#3}
\crefformat{cor}{Corollary~#2#1#3}
\crefformat{goal}{Goal~#2#1#3}
\crefformat{lem}{Lemma~#2#1#3}
\crefformat{FGL}{the Factor Group Lemma~#2(#1)#3}

\theoremstyle{definition}
\newtheorem*{notation}{Notation}

\newtheorem{note}[equation]{Note}
\newtheorem{defn}[equation]{Definition}

\theoremstyle{remark}
\newtheorem{rem}[equation]{Remark}
\newtheorem*{ack}{Acknowledgments}

\crefformat{rem}{Remark~#2#1#3}

\makeatletter
\renewcommand{\tocsection}[3]{%
  \indentlabel{\@ifnotempty{#2}{\ignorespaces#1 \S#2.\ }}#3}
\renewcommand{\tocsubsection}[3]{%
  \indentlabel{\@ifnotempty{#2}{\ignorespaces#1 \S#2.\ }}#3}
\def\@tocpagenum#1{}
\def\l@section{\@tocline{1}{0pt}{3em}{5pc}{}}
\def\l@subsection{\@tocline{2}{0pt}{4.5em}{5pc}{}} % 4.5em was smaller
\renewcommand{\@bibtitlestyle}{%
\par\bigbreak \centerline{\normalfont\scshape References}
}
\makeatother
\begin{document}

\title{Cayley graphs of order $27p$ are hamiltonian}

\author{Ebrahim Ghaderpour}
\address{Department of Mathematics and Computer Science,
University of Lethbridge, Lethbridge, Alberta, T1K~3M4, Canada}

\author{Dave Witte Morris}
\address{Department of Mathematics and Computer Science,
University of Lethbridge, Lethbridge, Alberta, T1K~3M4, Canada}

\begin{abstract}
Suppose $G$ is a finite group, such that $|G| = 27p$, where $p$~is prime. We show that if $S$ is any generating set of~$G$, then there is a hamiltonian cycle in the corresponding Cayley graph $\Cay(G;S)$.
\end{abstract}

\maketitle

\begin{thmstar} \label{27p}
If\/ $|G| = 27p $, where $p$~is prime, then every connected Cayley graph on~$G$ has a hamiltonian cycle.
\end{thmstar}

Combining this with results in \cite{CurranMorris2-16p,GhaderpourMorris-30p,M2Slovenian-LowOrder} establishes that:
	\begin{align}
	\begin{matrix}
	\text{\it Every Cayley graph on~$G$ has a hamiltonian cycle}
	\\ \text{if\/ $|G| = kp$, 
	where $p$~is prime, $1 \le k < 32$, and $k \neq 24$.}
	\end{matrix}
	\end{align}

The remainder of the paper provides a proof of the theorem. Here is an outline: 

\tableofcontents

\section{Preliminaries: known results on hamiltonian cycles in Cayley graphs}

For convenience, we record some known results that provide hamiltonian cycles in various Cayley graphs, after fixing some notation.

\begin{notation}[{}{\cite[\S1.1 and \S5.1]{Gorenstein-FinGrps}}]
For any group~$G$, we use:
	\begin{enumerate}
	\item $G'$ to denote the \emph{commutator subgroup} $[G,G]$ of~$G$,
	\item $Z(G)$ to denote the \emph{center} of~$G$,
	and
	\item $\Phi(G)$ to denote the \emph{Frattini subgroup} of~$G$.
	\end{enumerate}
For $a,b \in G$, we use $a^b$ to denote the \emph{conjugate} $b^{-1} a b$.
\end{notation}

\begin{notation}
If $(s_1,s_2,\ldots,s_n)$ is any sequence, we use $(s_1,s_2,\ldots,s_n)\#$ to denote the sequence $(s_1,s_2,\ldots,s_{n-1})$ that is obtained by deleting the last term.
\end{notation}

\begin{thm}[(Maru\v si\v c, Durnberger, Keating-Witte \cite{KeatingWitte})] \label{KeatingWitte}
 If $G'$ is a cyclic group of prime-power order, then
every connected Cayley graph on~$G$ has a hamiltonian cycle.
 \end{thm}

\begin{lem}[{}{\cite[Lem.~2.27]{M2Slovenian-LowOrder}}]
\label{NormalEasy}
 Let $S$ generate the finite group~$G$, and let $s \in S$. If
 \begin{itemize}
 \item $\langle s \rangle \normal G$,
 \item $\Cay \bigl( G/\langle s \rangle ; S \bigr)$ has a
hamiltonian cycle,
 and
 \item either
 \begin{enumerate}
 \item \label{NormalEasy-Z} 
 $s \in Z(G)$,
 or
 \item \label{NormalEasy-p}
 $|s|$ is prime,
 \end{enumerate}
 \end{itemize}
 then $\Cay(G;S)$ has a hamiltonian cycle.
 \end{lem}

\begin{lem}[{}{\cite[Lem.~2.7]{CurranMorris2-16p}}] \label{CyclicNormal2p} %% still correct ref???
 Let $S$ generate the finite group~$G$, and let $s \in S$. If
 \begin{itemize}
 \item $\langle s \rangle \normal G$,
 \item $|s|$ is a divisor of $pq$, where $p$ and~$q$ are distinct primes,
 \item $s^p \in Z(G)$,
 \item $|G/\langle s \rangle|$ is divisible by~$q$,
 and
 \item $\Cay \bigl( G/\langle s \rangle ; S \bigr)$ has a hamiltonian cycle,
 \end{itemize}
 then there is a hamiltonian cycle in $\Cay(G;S)$.
 \end{lem}

The following results are well known (and easy to prove):

\begin{FGL}[{(``Factor Group Lemma'')}] \label{FGL}
 Suppose
 \begin{itemize}
 \item $S$ is a generating set of~$G$,
 \item $N$ is a cyclic, normal subgroup of~$G$,
 \item $(s_1 N,\ldots,s_n N)$ is a hamiltonian cycle in $\Cay(G/N;S)$,
 and
 \item the product $s_1 s_2 \cdots s_n$ generates~$N$.
 \end{itemize}
 Then $(s_1,\ldots,s_n)^{|N|}$ is a hamiltonian cycle in $\Cay(G;S)$.
 \end{FGL}

\begin{cor} \label{DoubleEdge}
  Suppose
 \begin{itemize}
 \item $S$ is a generating set of~$G$,
 \item $N$ is a normal subgroup of~$G$, such that $|N|$ is prime,
 \item $s \equiv t \pmod{N}$ for some $s,t \in S \cup S^{-1}$ with $s \neq t$,
 and
 \item there is a hamiltonian cycle in $\Cay(G/N;S)$ that uses at least one edge labelled~$s$.
 \end{itemize}
 Then there is a hamiltonian cycle in $\Cay(G;S)$.
 \end{cor}

 \begin{defn}
 If $H$ is any subgroup of~$G$, then $H \backslash {\Cay(G;S)}$ denotes the multigraph in which:
 	\begin{itemize}
	\item the vertices are the right cosets of~$H$,
	and
	\item there is an edge joining $Hg_1$ and~$Hg_2$ for each $s \in S \cup S^{-1}$, such that $g_1 s \in H g_2$.
	\end{itemize}
Thus, if there are two different elements $s_1$ and~$s_2$ of $S \cup S^{-1}$, such that $g_1 s_1$ and $g_1 s_2$ are both in $H g_2$, then the vertices $Hg_1$ and $Hg_2$ are joined by a double edge.
 \end{defn}

\begin{lem}[{}{\cite[Cor.~2.9]{M2Slovenian-LowOrder}}] \label{MultiDouble}
  Suppose
 \begin{itemize}
 \item $S$ is a generating set of~$G$,
 \item $H$ is a subgroup of~$G$, such that $|H|$ is prime,
 \item the quotient multigraph $H \backslash {\Cay(G;S)}$ has a hamiltonian cycle~$C$,
 and
 \item $C$ uses some double-edge of~$H \backslash {\Cay(G;S)}$.
 \end{itemize}
 Then there is a hamiltonian cycle in $\Cay(G;S)$.
 \end{lem}

\begin{thm}[{}{\cite[Cor.~3.3]{Morris-2genNilp}}] \label{{pk}Subgrp}
  Suppose
 \begin{itemize}
 \item $S$ is a generating set of~$G$,
 \item $N$ is a normal $p$-subgroup of~$G$,
 and
 \item $s t^{-1} \in N$, for all $s,t \in S$.
 \end{itemize}
 Then $\Cay(G;S)$ has a hamiltonian cycle.
 \end{thm}

 \begin{rem} \label{>3}
 In the proof of our main result, we may assume $p \ge 5$, for otherwise either:
 	\begin{itemize}
	\item $|G| = 54$ is of the form $18q$, where $q$ is prime, so \cite[Prop.~9.1]{M2Slovenian-LowOrder} applies,
	or
	\item $|G| = 3^4$ is a prime power, so the main theorem of \cite{Witte-pn} applies.
	\end{itemize}
\end{rem}

\section{Assume the Sylow $p$-subgroup of~$G$ is normal}

\begin{notation}
Let:
\begin{itemize}
	\item $G$ be a group of order $27p$, where $p$ is prime, and $p \ge 5$ \csee{>3},
	\item $S$ be a minimal generating set for~$G$,
	\item $P \cong \integer_p$ be a Sylow $p$-subgroup of~$G$,
	\item $w$ be a generator of~$P$,
	and
	\item $Q$ be a Sylow $3$-subgroup of~$G$
\end{itemize}
\end{notation}

\begin{assump}
In this section, we assume that $P$ is a normal subgroup of~$G$.
\end{assump}

Therefore $G$ is a semidirect product:
	$$G = Q \ltimes P .$$
We may assume $G'$ is not cyclic of prime order (for otherwise \cref{KeatingWitte} applies). This implies $Q$ is nonabelian, and acts nontrivially on~$P$, so 
	$$ \text{$G' = Q' \times P$ is cyclic of order $3p$.} $$

\begin{notation} 
Since $Q$ is a $3$-group and acts nontrivially on $P \iso \integer_p$, we must have $p \equiv 1 \pmod{3}$. Thus, we may choose $r \in \integer$, such that
	$$ \text{$ r^3 \equiv 1 \pmod{p}$, \quad but $r \not\equiv 1 \pmod{p}$} .$$
Dividing $r^3 - 1$ by $r -1$, we see that 
	$$r^2 + r + 1 \equiv 0 \pmod{p} .$$
\end{notation}

\subsection{A lemma that applies to both of the possible Sylow $3$-subgroups}
There are only $2$ nonabelian groups of order~$27$, and we will consider them as separate cases, but, first, we cover some common ground. 

\begin{note}
Since $Q$ is a nonabelian group of order~$27$, and $G = Q \ltimes P \iso Q \ltimes \integer_p$, it is easy to see that
	$$ Q' = \Phi(Q) = Z(Q) = Z(G) = \Phi(G) .$$
\end{note}

\begin{lem} \label{S=2choices}
Assume 
	\begin{itemize}
	\item $s \in (S \cup S^{-1}) \cap Q$, such that $s$ does not centralize~$P$,
	and
	\item $c \in C_Q(P) \smallsetminus \Phi(Q)$.
	\end{itemize}
Then we may assume $S$ is either $\{s, cw\}$ or $\{s, c^2w\}$ or $\{s, scw\}$ or $\{s, sc^2 w\}$.
\end{lem}

\begin{proof}
Since $G/P \iso Q$ is a $2$-generated group of prime-power order, there must be an element~$a$ of~$S$, such that $\{s,a\}$ generates $G/P$. We may write
	$$ \text{$a = s^i c^j z w^k$,
	\quad
	 with $0 \le i \le 2$, $1 \le j \le 2$,  $z \in Z(Q)$, and $0 \le k < p$.} $$
Note that:
	\begin{itemize}
	\item By replacing~$a$ with its inverse if necessary, we may assume $i \in \{0,1\}$. 
	\item By applying an automorphism of~$G$ that fixes~$s$ and maps $c$ to $cz^j$, we may assume $z$~is trivial (since $(cz^j)^j = c^j z^{j^2} = c^jz$).
	\item By replacing $w$ with~$w^k$ if $k \neq 0$, we may assume $k \in \{0,1\}$.
	\end{itemize}
Thus,
 	$$ \text{$a = s^i c^j w^k$ \quad with $i , k \in \{0,1\}$ and $j \in \{1,2\}$.} $$

\setcounter{case}{0}

\begin{case}
Assume $k = 1$.
\end{case}
Then $\langle s, a \rangle = G$, so $S = \{s,a\}$. This yields the four listed generating sets.

\begin{case}
Assume $k = 0$.
\end{case}
Then $\langle s,a \rangle = Q$, and there must be a third element~$b$ of~$S$, with $b \notin Q$; after replacing $w$~with an appropriate power, we may write $b =  tw$ with $t \in Q$. We must have $t \in \langle s, \Phi(Q) \rangle$, for otherwise $\langle s, b \rangle = G$ (which contradicts the minimality of~$S$). Therefore 
	$$ \text{$t = s^{i'} z'$ with $0 \le i' \le 2$ and $z' \in \Phi(Q) = Z(G)$. } $$
We may assume:
	\begin{itemize}
	\item $i' \neq 0$, for otherwise $b = z' w \in S \cap \bigl( Z(G) \times P \bigr)$, so \cref{CyclicNormal2p} applies.
	\item $i' = 1$, by replacing $b$ with its inverse if necessary.
	\item $z' \neq e$, for otherwise $s$ and~$b$ provide a double edge in $\Cay(G/P; S)$, so \cref{DoubleEdge} applies.
	\end{itemize} 
Then $s^{-1} b = z' w$ generates $Z(G) \times P$. 

Consider the hamiltonian cycles 
	$$ \text{$(a^{-1},s^2)^3$, \quad $\bigl( (a^{-1},s^2)^3 \#, b \bigr)$, \quad and \quad $\bigl( (a^{-1},s^2)^3 \#\#, b^2 \bigr)$} $$
 in $\Cay(G/\langle z,w \rangle;S)$. 
 Letting $z'' = (a^{-1}s^2)^3 \in \langle z \rangle$, we see that their endpoints in~$G$ are (respectively):
 	$$\text{$z''$, \quad $z''(s^{-1} b) = z'' z' \, w$, \quad and \quad  $z''(s^{-1} b)^s (s^{-1} b) = z'' (z')^2 \, w^s w$} .$$
The final two endpoints both have a nontrivial projection to~$P$ (since $s$, being a $3$-element, cannot invert~$w$), and at least one of these two endpoints also has a nontrivial projection to $Z(G)$. Such an endpoint generates $Z(G) \times P = \langle z, w \rangle$, so \cref{FGL} provides a hamiltonian cycle in $\Cay(G;S)$.
\end{proof}

\subsection{Sylow $3$-subgroup of exponent~$3$}

\begin{lem} \label{exp3S}
Assume $Q$ is of exponent~$3$, so
	$$ Q = \langle \, x, y, z \mid x^3 = y^3 = z^3 = e, \ [x,y] = z, \ [x,z] = [y,z] = e  \,\rangle .$$
Then we may assume:
	\begin{enumerate}
	\item \label{exp3S-G}
	$w^x = w^r$, but $y$ and~$z$ centralize~$P$,
	and
	\item \label{exp3S-S}
	either:
		\begin{enumerate}
		\item $S = \{x, yw\}$,
		or
		\item $S = \{x, xyw\}$.
		\end{enumerate}
	\end{enumerate}
\end{lem}

\begin{proof}
\pref{exp3S-G} Since $Q$ acts nontrivially on~$P$, and $\Aut(P)$ is cyclic, but $Q/\Phi(Q)$ is not cyclic, there must be elements $a$ and~$b$ of $Q \smallsetminus \Phi(Q)$, such that $a$~centralizes~$P$, but $b$~does not. (And $z$ must centralize~$P$, because it is in $Q'$.) By applying an automorphism of~$Q$, we may assume $a = y$ and $b = x$. Furthermore, we may assume $w^x = w^r$ by replacing $x$ with its inverse if necessary.

\pref{exp3S-S} $S$ must contain an element that does not centralize~$P$, so we may assume $x \in S$. 
By applying \cref{S=2choices} with $s = x$ and $c = y$, we see that we may assume $S$ is:
	$$ \text{$\{x, yw\}$ or $\{x, y^2w\}$ or $\{x, xyw\}$ or $\{x, xy^2w\}$} . $$
But there is an automorphism of~$G$ that fixes $x$ and~$w$, and sends $y$ to~$y^2$, so we need only consider $2$ of these possibilities.
\end{proof}

\begin{prop}
Assume, as usual, that\/ $|G| = 27p$, where $p$ is prime, and that $G$ has a normal Sylow $p$-subgroup. If the Sylow $3$-subgroup $Q$ is of exponent\/~$3$, then\/ $\Cay(G;S)$ has a hamiltonian cycle.
\end{prop}

\begin{proof}
We write $\qt{\phantom{x}}$ for the natural homomorphism from~$G$ to $\qt{G} = G/P$.
From \fullcref{exp3S}{S}, we see that we need only consider two possibilities for~$S$.

\setcounter{case}{0}

\begin{case}
Assume $S = \{x, yw\}$.
\end{case}
For $a = x$ and $b = yw$, we have the following hamiltonian cycle in $\Cay(G/P;S)$:
\begin{align*}\begin{array}{ccccccccccccc}
&& \qt{e}  \bya {x } \bya {x^2 } \byb {x^2 y } \byai { xyz} \byai {yz^2}
\\
  \byb { y^2 z^2} \byb {z^2 } \bya {xz^2 } \bya {x^2z^2 } \byb {x^2yz^2 } \bya { yz}
   \\
    \bya {xy } \byb {xy^2} \bya {x^2y^2z } \byb { x^2z} \byb {x^2yz } \byai { xyz^2 }
    \\
  \byb {xy^2z^2 } \bya {x^2y^2 } \bya { y^2z} \byb { z} \bya {xz } \bybi { xy^2z}
  \\
   \bya { x^2y^2z^2} \bya {y^2 } \bybi { y} \bybi {e}
 \end{array}\end{align*}
Its endpoint in~$G$ is
	\begin{align*}
	a^2 b a^{-2} & b^2a^2 b a^2 b a b^2 a^{-1} b a^2 b a b^{-1} a^2 b^{-2}
	\\&=
	x^2 yw x^{-2} (yw)^2x^2 yw x^2 yw x (yw)^2 x^{-1} yw x^2 yw x (yw)^{-1} x^2 (yw)^{-2}
	\\&= 
	x^2 yw x y^2w^2 x^2 yw x^2 yw x y^2w^2 x^2 yw x^2 yw x y^2w^{-1} x^2 yw^{-2}
	. \end{align*}
Since the walk is a hamiltonian cycle in $G/P$, we know that this endpoint is in $P = \langle w \rangle$. So all terms except powers of~$w$ must cancel. Thus, we need only calculate the contribution from each appearance of~$w$ in this expression. To do this, note that if a term~$w^i$ is followed by a net total of $j$~appearances of~$x$, then the term contributes a factor of $w^{i r^j}$ to the product. So the endpoint in~$G$ is:
	$$ w^{r^{13}} w^{2r^{12}} w^{r^{10}} w^{r^8} w^{2r^7} w^{r^5} w^{r^3} w^{-r^2} w^{-2} . $$
Since $r^3 \equiv 1 \pmod{p}$, this simplifies to
	\begin{align*}
	w^r w^{2} w^{r} w^{r^2} w^{2r} w^{r^2} w w^{-r^2} w^{-2}
	&=  
	w^{r+2+r+r^2+2r+r^2+1-r^2-2}
	\\&= 
	w^{r^2+4r+1}
	= 
	w^{r^2+r+1}w^{3r}
	= w^0 w^{3r}
	= w^{3r}
	.
	\end{align*}
Since $p \nmid 3r$, this endpoint generates $P$, so \cref{FGL} provides a hamiltonian cycle in $\Cay(G;S)$.

\begin{case}
Assume $S = \{x, xyw\}$.
\end{case}
For $a = x$ and $b = xyw$, we have the hamiltonian cycle
	$$ \bigl( (a,b^2)^3\#,a \bigr)^3 $$
in $\Cay(G/P;S)$. Its endpoint in~$G$ is
	\begin{align*}
	 \bigl( (ab^2)^3 \, b^{-1}a \bigr)^3
	&=  \Bigl( \bigl( x(xyw)^2 \bigr)^3 \, (xyw)^{-1}x \Bigr)^3
	=  \Bigl( \bigl( x(x^2 y^2 w^{r+1}) \bigr)^3 \, (w^{-1} y^{-1} x^{-1}) x \Bigr)^3
	\\&=  \Bigl( \bigl( y^2 w^{r+1} \bigr)^3 \, (w^{-1} y^{-1} )  \Bigr)^3
	=  \Bigl( w^{3(r+1)}  \, (w^{-1} y^{-1} )  \Bigr)^3
	=  \Bigl( y^{-1} w^{3r+2}   \Bigr)^3
	\\&= w^{3(3r+2)}
	. \end{align*}
Since we are free to choose~$r$ to be either of the two primitive cube roots of~$1$ in $\integer_p$, and the equation $3r + 2 = 0$ has only one solution in~$\integer_p$, we may assume $r$~has been selected to make the exponent nonzero. Then \cref{FGL} provides a hamiltonian cycle in $\Cay(G;S)$.
\end{proof}

\subsection{Sylow $3$-subgroup of exponent~$9$}

\begin{lem} \label{GensExp9}
Assume $Q$ is of exponent~$9$, so
	$$ Q = \langle\, x, y \mid x^9 = y^3 = e, \ [x,y] = x^3 \,\rangle .$$
There are two possibilities for~$G$, depending on whether $C_Q(P)$ contains an element of order~$9$ or not.
\begin{enumerate}
	\item  \label{Exp9-no9}
Assume $C_Q(P)$ does not contain an element of order~$9$. Then we may assume $y$ centralizes~$P$, but $w^x = w^r$. Furthermore, we may assume:
	\begin{enumerate}
		\item \label{Exp9-no9-x+yw}
		$S = \{x, yw \} $, 
		or
		\item \label{Exp9-no9-x+xyw}
		$S = \{x, xyw\}$.
	\end{enumerate}
	\item \label{Exp9-9cent}
Assume $C_Q(P)$ contains an element of order~$9$. Then we may assume $x$ centralizes~$P$, but $w^y = w^r$. Furthermore, we may assume:
	\begin{enumerate}
		\item \label{Exp9-9cent-xw+y}
		$S = \{xw, y\} $,
		\item \label{Exp9-9cent-xyw+y}
		$S = \{xyw, y\}$,
		\item \label{Exp9-9cent-xy+xw}
		$S = \{xy,xw\}$,
		or
		\item \label{Exp9-9cent-xy+x2yw}
		$S = \{xy, x^2yw\}$.
	\end{enumerate}
\end{enumerate}
\end{lem}

\begin{proof}
(\ref{Exp9-no9})
Since $x$ has order~$9$, we know that it does not centralize~$P$. But $x^3$ must centralize~$P$ (since $x^3$ is in~$G'$). Therefore, we may assume $w^x = x^r$ (by replacing $x$ with its inverse if necessary). Also, since $Q/C_Q(P)$ must be cyclic (because $\Aut(P)$ is cyclic), but $C_G(P)$ does not contain an element of order~$9$, we see that $C_Q(P)$ contains every element of order~$3$, so $y$ must be in $C_Q(P)$.

 Since $S$ must contain an element that does not centralize~$P$, we may assume $x \in S$. 
By applying \cref{S=2choices} with $s = x$ and $c = y$, we see that we may assume $S$ is:
	$$ \text{$\{x, yw\}$ or $\{x, y^2w\}$ or $\{x, xyw\}$ or $\{x, xy^2w\}$} . $$
The second generating set need not be considered, because $(y^2w)^{-1} = y w^{-1} = yw'$, so it is equivalent to the first. Also, the fourth generating set can be converted into the third, since there is an automorphism of~$G$ that fixes~$y$, but takes $x$ to $x y w$ and $w$ to~$w^{-1}$.

\medskip
(\ref{Exp9-9cent}) We may assume $x \in C_Q(P)$, so $C_Q(P) = \langle x \rangle$.

We know that $S$ must contain an element~$s$ that does not centralize~$P$, and there are two possibilities: either 
	\begin{enumerate} 
	\renewcommand{\theenumi}{\Roman{enumi}}
	\item \label{GensExp9Pf-9cent-3}
 $s$~has order~$3$,
	or
	\item \label{GensExp9Pf-9cent-9}
$s$~has order~$9$.
	\end{enumerate}
We consider these two possibilities as separate cases.

\begingroup

 \renewcommand{\thecase}{\ref{GensExp9Pf-9cent-3}}
\begin{case}
Assume $s$~has order\/~$3$.
\end{case}
We may assume $s = y$.
Letting $c = x$, we see from \cref{S=2choices} that we may assume $S$ is either 
	$$ \text{$\{y, xw\}$ or $\{y, x^2w\}$ or $\{y, yxw\}$ or $\{y, yx^2w\}$.} $$
The second and fourth generating sets need not be considered, because there is an automorphism of~$G$ that fixes $y$ and~$w$, but takes $x$ to~$x^2$. Also, the third generating set may be replaced with $\{y,xyw\}$, since there is an automorphism of~$G$ that fixes~$y$ and~$w$, but takes $x$ to~$y^{-1} x y$.

 \renewcommand{\thecase}{\ref{GensExp9Pf-9cent-9}}
\begin{case}
Assume $s$~has order\/~$9$.
\end{case}
We may assume $s = xy$.
Letting $c = x$, we see from \cref{S=2choices} that we may assume $S$ is either 
	$$ \text{$\{xy, xw\}$ or $\{xy, x^2w\}$ or $\{xy, xyxw\}$ or $\{xy, xyx^2w\}$.} $$
The second generating set is equivalent to $\{xy, xw\}$, since the automorphism of~$G$ that sends $x$ to~$x^4$, $y$ to~$x^{-3}y$, and $w$ to~$w^{-1}$ maps it to $\{xy, (xw)^{-1}\}$.
The third generating set is mapped to $\{xy, x^2yw\}$ by the automorphism that sends $x$ to~$x [x,y]$ and $y$ to~$[x,y]^{-1}y$.
The fourth generating set need not be considered, because $xyx^2 w$ is an element of order~$3$ that does not centralize~$P$, which puts it in the previous case.
\endgroup
\end{proof}

\begin{prop}
Assume, as usual, that\/ $|G| = 27p$, where $p$ is prime, and that $G$ has a normal Sylow $p$-subgroup. If the Sylow $3$-subgroup $Q$ is of exponent\/~$9$, then\/ $\Cay(G;S)$ has a hamiltonian cycle.
\end{prop}

\begin{proof}
We will show that, for an appropriate choice of $a$ and~$b$ in~$S \cup S^{-1}$, the walk
	\begin{align} \label{HC}
	 (a^3, b^{-1}, a, b^{-1}, a^4, b^2, a^{-2}, b, a^2, b, a^3, b, a^{-1}, b^{-1}, a^{-1}, b^{-2} ) 
	 \end{align}
provides a hamiltonian cycle in $\Cay(G/P;S)$ whose endpoint in~$G$ generates~$P$ (so \cref{FGL} applies).

We begin by verifying two situations in which \pref{HC} is a hamiltonian cycle:

\begin{enumerate}
\renewcommand{\theenumi}{HC\arabic{enumi}}
\item \label{HC3}
If $|\qtit{a}| = 9$, $|\qtit{b}| = 3$, and $\qtit{a^b} = \qtit{a^4}$ in $\qtit{G} = G/P$, then we have the hamiltonian cycle
$$\begin{array}{ccccccccccccccc}
&&\qtit{e}
\bya {a}
\bya {a^2}
\bya {a^3}
\bybi {a^3b^2}
\bya {a^7b^2}
\bybi {a^7b}
\\
\bya {a^5b}
\bya {a^3b}
\bya {ab}
\bya {a^8b}
\byb {a^8b^2}
\byb {a^8}
\byai {a^7}
\\
\byai {a^6}
\byb {a^6b}
\bya {a^4b}
\bya {a^2b}
\byb {a^2b^2}
\bya {a^6b^2}
\bya {ab^2}
\\
\bya {a^5b^2}
\byb {a^5}
\byai {a^4}
\bybi {a^4b^2}
\byai {b^2}
\bybi {b}
\bybi {e}
\end{array}$$

\item  \label{HC9}
If $|\qtit{a}| = 9$, $|\qtit{b}| = 9$, $\qtit{a^b} = \qtit{a^7}$, and $\qtit{b^3} = \qtit{a^6}$ in $\qtit{G} = G/P$, then we have the hamiltonian cycle
$$\begin{array}{ccccccccccccccc}
&&\qtit{e}
\bya {a}
\bya {a^2}
\bya {a^3}
\bybi {a^6b^2}
\bya {a^4b^2}
\bybi {a^4b}
\\
\bya {a^8b}
\bya {a^3b}
\bya {a^7b}
\bya {a^2b}
\byb {a^2b^2}
\byb {a^8}
\byai {a^7}
\\
\byai {a^6}
\byb {a^6b}
\bya {ab}
\bya {a^5b}
\byb {a^5b^2}
\bya {a^3b^2}
\bya {ab^2}
\\
\bya {a^8b^2}
\byb {a^5}
\byai {a^4}
\bybi {a^7b^2}
\byai {b^2}
\bybi {b}
\bybi {e}
\end{array}$$
\end{enumerate}

To calculate the endpoint in~$G$, fix $r_1,r_2 \in \integer_p$, with
	$$ \text{$w^a = w^{r_1}$ and $w^b = w^{r_2}$} ,$$
and write
	$$ \text{$a = \ul a w_1$ and $b = \ul b w_2$, where $\ul a, \ul b \in Q$ and $w_1,w_2 \in P$.} $$
Note that if an occurrence of~$w_i$ in the product is followed by a net total of $j_1$~appearances of~$\ul a$ and a net total of $j_2$~appearances of~$\ul b$, then it contributes a factor of $w_i^{r_1^{j_1}r_2^{j_2}}$ to the product. 
(A similar occurrence of $w_i^{-1}$ contributes a factor of $w_i^{-r_1^{j_1}r_2^{j_2}}$ to the product.)
Furthermore, since $r_1^3 \equiv r_2^3 \equiv 1 \pmod{p}$, there is no harm in reducing $j_1$ and~$j_2$ modulo~$3$.

We will apply these considerations only in a few particular situations:
\begin{enumerate} \label{endpt} 
\renewcommand{\theenumi}{E\arabic{enumi}}
 \setlength{\itemsep}{\smallskipamount}

	\item  \label{endpt-b=w}
	Assume $w_1 = e$ (so $a \in Q$ and $\ul a = a$).  Then the endpoint of the path in~$G$ is 
	\begin{align*}
	 a^3 & b^{-1} a b^{-1} a^4 b^2 a^{-2} b a^2 b a^3 b a^{-1} b^{-1} a^{-1} b^{-2} 
	\\&= a^3 (\ul b w_2)^{-1} a (\ul b w_2)^{-1} a^4 (\ul b w_2)^2 a^{-2} (\ul b w_2) a^2 
	\\& \qquad \times (\ul b w_2) a^3 (\ul b w_2) a^{-1} (\ul b w_2)^{-1} a^{-1} (\ul b w_2)^{-2} 
	\\&= a^3 ( w_2^{-1} \ul b^{-1}) a ( w_2^{-1} \ul b^{-1}) a^4 (\ul b w_2\ul b w_2) a^{-2} (\ul b w_2)  a^2
	\\& \qquad \times  (\ul b w_2) a^3 (\ul b w_2) a^{-1} ( w_2^{-1} \ul b^{-1}) a^{-1} (w_2^{-1} \ul b^{-1}w_2^{-1} \ul b^{-1})
	. \end{align*}
By the above considerations, this simplifies to $w_2^m$, where
	\begin{align*}
	m
	&= -1 - r_1^2 r_2 + r_1r_2 + r_1 + r_2^2 + r_1r_2 + r_1 - r_1^2 - r_2 - r_2^2
	\\&= -r_1^2r_2 - r_1^2 + 2r_1r_2 + 2r_1 - r_2 - 1
	.\end{align*}
Note that:
	\begin{enumerate}
 	\setlength{\itemsep}{\smallskipamount}
	\item  \label{endpt-b=w-r2=1}
	If $r_1 \neq 1$ and $r_2 = 1$, then $m$ simplifies to $6r_1$, because $r_1^2 + r_1 + 1 \equiv 0 \pmod{p}$ in this case.
	\item  \label{endpt-b=w-not1}
	If $r_1 \neq 1$ and $r_2 \neq 1$, then $m$ simplifies to 
		$3r_1(r_2+1)$,
	because $r_1^2 + r_1 + 1 \equiv r_2^2 + r_2 + 1 \equiv 0 \pmod{p}$ in this case.
	\end{enumerate}

	\item  \label{endpt-a=w}
	Assume $w_2 = e$ (so $b \in Q$ and $\ul b = b$).  Then the endpoint of the path in~$G$ is 
	\begin{align*}
	 a^3 & b^{-1} a b^{-1} a^4 b^2 a^{-2} b a^2 b a^3 b a^{-1} b^{-1} a^{-1} b^{-2} 
	\\&= (\ul a w_1)^3 b^{-1} (\ul a w_1) b^{-1} (\ul a w_1)^4 b^2 (\ul a w_1)^{-2} b (\ul a w_1)^2 b (\ul a w_1)^3 b (\ul a w_1)^{-1} b^{-1} (\ul a w_1)^{-1} b^{-2} 
	\\&= (\ul a w_1 \ul a w_1 \ul a w_1) b^{-1} (\ul a w_1) b^{-1} (\ul a w_1 \ul a w_1 \ul a w_1 \ul a w_1) b^2 (w_1^{-1} \ul a^{-1} w_1^{-1} \ul a^{-1} ) 
		\\ &\hskip 1in 
		\times b (\ul a w_1 \ul a w_1) b (\ul a w_1 \ul a w_1 \ul a w_1) b (w_1^{-1} \ul a^{-1}) b^{-1} (w_1^{-1} \ul a^{-1}) b^{-2} 	
	. \end{align*}
By the above considerations, this simplifies to $w_1^m$, where
	\begin{align*}
	m
	&= r_1^2 + r_1 + 1 +r_1^2 r_2+ r_1 r_2^2+  r_2^2+ r_1^2 r_2^2+ r_1 r_2^2   -r_1 
		\\& \qquad {}   -r_1^2 + r_1^2r_2^2 +  r_1r_2^2 +  r_2 + r_1^2r_2 + r_1r_2 -r_1 - r_1^2 r_2
	\\&= 2r_1^2r_2^2 + 3r_1r_2^2 + r_2^2 + r_1^2r_2 + r_1r_2 + r_2 - r_1 + 1
	.\end{align*}
Note that:
	\begin{enumerate}
	 \setlength{\itemsep}{\smallskipamount}
	\item  \label{endpt-a=w-r1=1}
	If $r_1 = 1$ and $r_2 \neq 1$, then $m$ simplifies to $-3(r_2 + 2)$, because $r_2^2 + r_2 + 1 \equiv 0 \pmod{p}$ in this case.
	\item  \label{endpt-a=w-not1}
	If $r_1 \neq 1$ and $r_2 \neq 1$, then $m$ simplifies to 
		$-r_1r_2 - 2r_1 + r_2 + 2$,
	because $r_1^2 + r_1 + 1 \equiv r_2^2 + r_2 + 1 \equiv 0 \pmod{p}$ in this case.
	\end{enumerate}

\end{enumerate}

Now we provide a hamiltonian cycle for each of the generating sets listed in \cref{GensExp9}:

\begin{enumerate}
 \setlength{\itemsep}{\smallskipamount}

	\item[\pref{Exp9-no9-x+yw}]
	If $C_Q(P)$ has exponent~$3$, and $S = \{x, yw\}$, we let $a = x$ and $b = yw$ in \pref{HC3}. 
In this case, we have 
$w_1 = e$, $r_1 = r$, and $r_2 = 1$, 
so \pref{endpt-b=w-r2=1} tells us that the endpoint in~$G$ is
$w_2^{6r}$.
	
	\item[\pref{Exp9-no9-x+xyw}]
	If $C_Q(P)$ has exponent~$3$, and $S = \{x, xyw\}$, we let $a = x$ and $b = (xyw)^{-1}$ in \pref{HC9}. In this case, we have $w_1 = e$,
$r_1 = r$ and $r_2 = r^{-1} = r^2$, 
so \pref{endpt-b=w-not1} tells us that the endpoint in~$G$ is~$w_2^m$, where
	$$m = 3r_1(r_2+1) = 3r(r^2+1) = 3(r^3+r) \equiv  3(1 +r) = 3(r+1) \pmod{p}.$$
	
	\item[\pref{Exp9-9cent-xw+y}]
	If $C_Q(P)$ has exponent~$9$, and $S = \{xw, y\}$, we let $a = xw$ and $b = y$ in \pref{HC3}. In this case, we have 
$w_2 = e$, $r_1 = 1$ and $r_2 = r$, 
so \pref{endpt-a=w-r1=1} tells us that the endpoint in~$G$ is
$w_1^{-3(r+2)}$.
	
	\item[\pref{Exp9-9cent-xyw+y}]
	If $C_Q(P)$ has exponent~$9$, and $S = \{xyw, y\}$, we let $a = xyw$ and $b = y$ in \pref{HC3}. In this case, we have $w_2 = e$ and
$r_1 = r_2 = r$, 
so \pref{endpt-a=w-not1} tells us that the endpoint in~$G$ is
$w_2^m$, where
	$$ m = -r_1r_2 - 2r_1 + r_2 + 2 =-r^2 - 2r + r + 2 = -(r^2 + r + 1) + 3 \equiv 3 \pmod{p} .$$
	
	\item[\pref{Exp9-9cent-xy+xw}]
	If $C_Q(P)$ has exponent~$9$, and $S = \{xy,xw\}$, we let $a = xw$ and $b = (xy)^{-1}$ in \pref{HC9}. In this case, we have 
$w_2 = e$, $r_1 = 1$, and $r_2 = r^{-1} = r^2$, 
so \pref{endpt-a=w-r1=1} tells us that the endpoint in~$G$ is $w_1^m$, where
	$$m = -3(r_2 + 2) = -3(r^2 + 2) \equiv -3 \bigl( -( r + 1) + 2 \bigr) = 3(r-1) \pmod{p} .$$
	
	\item[\pref{Exp9-9cent-xy+x2yw}]
	If $C_Q(P)$ has exponent~$9$, and $S = \{xy, x^2yw\}$, we let $a = xy$ and $b = x^2yw$ in \pref{HC9}. In this case, we have $w_1 = e$ and
$r_1 = r_2 = r$, 
so \pref{endpt-b=w-not1} tells us that the endpoint in~$G$ is $w_2^m$, where
	$$m = 3r_1(r_2+1)  = 3r(r+1) = 3(r^2 + r) \equiv  3(-1) = -3  \pmod{p}.$$

	\end{enumerate}
In all cases, there is at most one nonzero value of~$r$ (modulo~$p$) for which the exponent of~$w_i$ is~$0$. Since we are free to choose~$r$ to be either of the two primitive cube roots of~$1$ in $\integer_p$, we may assume $r$~has been selected to make the exponent nonzero. Then \cref{FGL} provides a hamiltonian cycle in $\Cay(G;S)$.
\end{proof}

\section{Assume the Sylow $p$-subgroups of~$G$ are not normal}

\begin{lem} \label{PnotnormalIs13}
Assume
	\begin{itemize}
	\item $|G| = 27p$, where $p$ is an odd prime,
	and
	\item the Sylow $p$-subgroups of~$G$ are not normal. 
	\end{itemize}
Then $p = 13$, and $G = \integer_{13} \ltimes (\integer_3)^3$, where a generator~$w$ of\/~$\integer_{13}$ acts on\/ $(\integer_3)^3$ via multiplication on the right by the matrix
	$$ W = \begin{bmatrix}
	0 & 1 & 0 \\
	0 & 0 & 1 \\
	1 & 1 & 0 
	\end{bmatrix} .$$
Furthermore, we may assume 
	$$ \text{$S$ is of the form $\{w^i, w^j v\}$,} $$
where $v = (1,0,0) \in (\integer_3)^3$, and 
	$$(i,j) \in \{ (1,0), (2,0),  (1,2), (1,3), (1,5), (1,6), (2,5) \} .$$
\end{lem}

\begin{proof}
Let $P$ be a Sylow $p$-subgroup of~$G$, and $Q$~be a Sylow $3$-subgroup of~$G$.
Since no odd prime divides $3-1$ or $3^2 - 1$, and $13$ is the only odd prime that divides $3^3 - 1$, Sylow's Theorem \cite[Thm.~15.7, p.~230]{Judson-AlgText} implies that $p = 13$, and that $N_G(P) = P$, so $G$ must have a normal $p$-complement \cite[Thm.~7.4.3]{Gorenstein-FinGrps}; i.e., $G = P \ltimes Q$. Since $P$ must act nontrivially on~$Q$ (since $P$ is not normal), we know that it must act nontrivially on $Q/\Phi(Q)$ \cite[Thm.~5.3.5, p.~180]{Gorenstein-FinGrps}.
However, $P$ cannot act nontrivially on an elementary abelian group of order $3$ or~$3^2$, because $|P| = 13$ is not a divisor of $3-1$ or $3^2 - 1$. 
Therefore, we must have $|Q/\Phi(Q)| = 3^3$, so $Q$ must be elementary abelian (and the action of~$P$ is irreducible). 

Let $W$ be the matrix representing the action of~$w$ on $ (\integer_3)^3$ (with respect to some basis that will be specified later).
In the polynomial ring $\integer_3[X]$, we have the factorization:
	\begin{align} \label{Factor(x13-1)}
	 \frac{X^{13} - 1}{X - 1} &=  (X^{3} - X - 1) \cdot (X^{3} + X^{2} - 1) 
	 \\ & \qquad \notag
	 \cdot (X^{3} +
X^{2} + X - 1) \cdot (X^{3} - X^{2} - X - 1) 
	. \end{align}
Since $w^{13} = e$, the minimal polynomial of~$W$ must be one of the factors on the right-hand side. By replacing $w$ with an appropriate power, we may assume it is the first factor. Then, choosing any nonzero $v \in (\integer_3)^3$, the matrix representation of~$w$ with respect to the basis $\{v, v^w, v^{w^2} \}$ is~$W$ (the Rational Canonical Form).

Now, let $\zeta$ be a primitive $13$th root of unity in the finite field $\GF(27)$. Then any Galois automorphism of $\GF(27)$ over $\GF(3)$ must raise $\zeta$ to a power. Since the subgroup of order~$3$ in $\integer_{13}^\times$ is generated by the number~$3$, we conclude that the orbit of~$\zeta$ under the Galois group is $\{\zeta, \zeta^3, \zeta^9\}$. These must be the $3$ roots of one of the irreducible factors on the right-hand side of \pref{Factor(x13-1)}. Thus, for any $k \in \integer_{13}^\times$, the matrices $W^k$, $W^{3k}$, and~$W^{9k}$ all have the same minimal polynomial, so they are conjugate under $\GL_3(3)$. That is: 
\begin{align} \label{MinPolysOfZ13xQ}
\begin{matrix}
\text{powers of $W$ in the same row of the}
\\ \text{following table are conjugate under $\GL_3(3)$:}
\end{matrix}
\qquad
\begin{array}{|c|}
\noalign{\hrule}
W, W^3, W^9  \\
\noalign{\hrule}
W^2, W^5, W^6  \\
\noalign{\hrule}
W^4, W^{12}, W^{10} \\
\noalign{\hrule}
W^7, W^8, W^{11} \\
\noalign{\hrule}
\end{array}
\end{align}

There is an element~$a$ of~$S$ that generates $G/Q \iso P$. Then $a$ has order~$p$, so, replacing it by a conjugate, we may assume $a \in P = \langle w \rangle$, so $a = w^i$ for some $i \in \integer_{13}^\times$. From \pref{MinPolysOfZ13xQ}, we see that we may assume $i \in \{1,2\}$ (perhaps after replacing $a$ by its inverse).

Now let $b$ be the second element of~$S$, so we may assume $b = w^j v$ for some~$j$. We may assume $0 \le j \le 6$ (by replacing $b$~with its inverse, if necessary). We may also assume $j \neq i$, for otherwise $S \subset a Q$, so \cref{{pk}Subgrp} applies.

If $j = 0$, then $(i,j)$ is either $(1,0)$ or $(2,0)$, both of which appear in the list; henceforth, let us assume $j \neq 0$.

\setcounter{case}{0}

\begin{case}
Assume $i = 1$.
\end{case}
Since $j \neq i$, we must have $j \in \{2,3,4,5,6\}$. 

Note that, since $W^3$ is conjugate to~$W$ under $\GL_3(3)$ (since they are in the same row of \pref{MinPolysOfZ13xQ}), we know that the pair $(w, w^4)$ is isomorphic to the pair $\bigl( w^3, (w^3)^4 \bigr) = (w^3, w^{-1})$. By replacing $b$ with its inverse, and then interchanging $a$ and~$b$, this is transformed to $(w,w^3)$. So we may assume $j \neq 4$.

\begin{case}
Assume $i = 2$.
\end{case}
We may assume $W^j$ is in the second or fourth row of the table (for otherwise we could interchange $a$ with~$b$ to enter the previous case. So $j \in \{2,5,6\}$. Since $j \neq i$, this implies $j \in \{5,6\}$. However, since $W^5$ is conjugate to~$W^2$ (since they are in the same row of \pref{MinPolysOfZ13xQ}), and we have $(w^2)^3 = w^6$ and $(w^5)^3 = w^2$, we see that the pair $(w^2, w^6)$ is isomorphic to $(w^2, w^5)$. So we may assume $j \neq 6$.
\end{proof}

\begin{prop}
If\/ $|G| = 27p$, where $p$ is prime, and the Sylow $p$-subgroups of~$G$ are not normal, then $\Cay(G;S)$ has  a hamiltonian cycle.
\end{prop}

\begin{proof}
From \cref{PnotnormalIs13} (and \cref{>3}), we may assume $G = \integer_{13} \ltimes (\integer_3)^3$. For each of the generating sets listed in \cref{PnotnormalIs13}, we provide an explicit hamiltonian cycle in the quotient multigraph $P \backslash {\Cay(G;S)}$ that uses at least one double edge. So \cref{MultiDouble} applies.

To save space, we use $i_1i_2i_3$ to denote the vertex $P (i_1, i_2,i_3)$.

\begingroup \renewcommand{\qtit}{} % don't overline the vertices
$$\begin{array}{ccccccccccccccc}
 \noalign{$(i,j) = ( 1 , 0 )$ \quad $a = w , \quad a^{-1} = w^{12}, \quad b = (1,0,0), \quad b^{-1} = (-1,0,0)$\smallskip}
\noalign{Double edge: $ 222 \rightarrow 022 $ with $ a^{-1} $ and $ b $\medskip} 
&& \qtit{000}
\bybi {200} \bya {020} \bya {002} \bya {220} \bybi {120} \bya {012}
\\ 
\bya {221} \bya {102} \byb {202} \bya {210} \bya {021} \bya {112} \bya {201}
\\
\bybi {101} \byai {211} \byai {212} \byai {222} \byb {022} \byb {122} \byai {121} 
\\
\byai {111} \bybi {011} \byai {110} \byai {001} \byai {010} \byai {100} \bybi {000} 
 \end{array}$$

$$\begin{array}{ccccccccccccccccc}
\noalign{$(i,j) = ( 2 , 0 )$ \quad $a = w^2 , \quad a^{-1} = w^{11}, \quad b = (1,0,0), \quad b^{-1} = (-1,0,0)$\smallskip}
\noalign{Double edge: $ 020 \rightarrow 220 $ with $ a $ and $ b^{-1} $\medskip} 
&&{000}
\bybi {200} \bya {002} \bya {022} \bya {212} \bybi {112} \byai {210}
\\
\byai {122} \byai {111} \byai {110} \bybi {010} \byai {201} \bybi {101} \bya {012} 
\\
\bya {102} \bya {020} \bybi {220} \bya {222} \bya {211} \bya {120} \bya {221} 
\\
\byb {021} \byai {202} \byai {121} \byai {011} \byai {001} \byai {100} \bybi {000}  
  \end{array}$$

$$\begin{array}{ccccccccccccccccc}
\noalign{$(i,j) = ( 1 , 2 )$ \quad $a = w , \quad a^{-1} = w^{12}, \quad b = w^2(1,0,0), \quad b^{-1} = w^{11}(-1,-1,1)$\smallskip}
\noalign{Double edge: $ 220 \rightarrow 022 $ with $ a $ and $ b $\medskip} 
&& {000}
\bybi {221} \byai {012} \byai {120} \bybi {102} \bybi {200} \bya {020} 
\\
\bya {002} \bya {220} \byb {022} \bya {222} \byb {011} \bya {111} \bya {121} 
\\
\bya {122} \bya {202} \bya {210} \bya {021} \bya {112} \bybi {101} \byai {211} 
\\\byai {212} \byb {201} \byb {110} \byai {001} \byai {010} \byai {100} \bybi {000}  
 \end{array}$$

$$\begin{array}{ccccccccccccccccc}
\noalign{$(i,j) = ( 1 , 3 )$ \quad $a = w , \quad a^{-1} = w^{12}, \quad b = w^3 (1,0,0), \quad b^{-1} = w^{10}(0,1,-1)$\smallskip}
\noalign{Double edge: $ 200 \rightarrow 020 $ with $ a $ and $ b $\medskip} 
&& {000}
\bybi {012} \byai {120} \bybi {221} \bya {102} \bya {200} \byb {020} 
\\
\bya {002} \bya {220} \bya {022} \bya {222} \bya {212} \bya {211} \bya {101} 
\\
\bybi {201} \byai {112} \byai {021} \byai {210} \byai {202} \byai {122} \byb {121}
\\
\byai {111} \byai {011} \byai {110} \byai {001} \byai {010} \byai {100} \bybi {000}
  \end{array}$$

$$\begin{array}{ccccccccccccccccc}
\noalign{$(i,j) = ( 1 , 5 )$ \quad $a = w , \quad a^{-1} =  w^{12}, \quad b = w^5 (1,0,0), \quad b^{-1} = w^8 (1,0,1)$\smallskip}
\noalign{Double edge: $ 220 \rightarrow 022 $ with $ a $ and $ b^{-1} $\medskip} 
&&{000}
\bybi {101} \bya {120} \bya {012} \bya {221} \bybi {010} \bya {001} 
\\
\bya {110} \bya {011} \bya {111} \byb {121} \bya {122} \bybi {102} \bya {200} 
\\
\bya {020} \bya {002} \bya {220} \bybi {022} \bya {222} \bya {212} \bya {211} 
\\
\byb {202} \bya {210} \bya {021} \bya {112} \bya {201} \bya {100} \bybi {000}  
 \end{array}$$

$$\begin{array}{ccccccccccccccccc}
\noalign{$(i,j) = ( 1 , 6 )$ \quad $a = w , \quad a^{-1} =  w^{12}, \quad b = w^6 (1,0,0), \quad b^{-1} = w^7 (-1,1,1)$\smallskip}
\noalign{Double edge: $ 021 \rightarrow 210 $ with $ a^{-1} $ and $ b $\medskip} 
&& {000}
\bybi {211} \bybi {201} \byai {112} \byai {021} \byb {210} \byb {101} 
\\
\byb {120} \bya {012} \bya {221} \bya {102} \bya {200} \bya {020} \bya {002} 
\\
\bya {220} \bya {022} \bya {222} \bya {212} \byb {202} \byai {122} \byai {121} 
\\
\byai {111} \byai {011} \byai {110} \byai {001} \byai {010} \byai {100} \bybi {000}  
 \end{array}$$

\begin{align*}
\raise3\baselineskip \hbox{$     % adjust height to put QED box in correct place
\begin{array}{ccccccccccccccccc}
\noalign{$(i,j) = ( 2 , 5 )$ \quad $a = w^2 , \quad a^{-1} =  w^{11}, \quad b = w^5 (1,0,0), \quad b^{-1} = w^8 (1,0,1)$\smallskip}
\noalign{Double edge: $ 112 \rightarrow 210 $ with $ a^{-1} $ and $ b $\medskip} 
&&{000}
\bybi {101} \bya {012} \byb {102} \bya {020} \bya {220} \bya {222} 
\\
\byb {112} \byb {210} \byai {122} \byai {111} \byai {110} \byai {010} \byai {201} 
\\
\byai {021} \byai {202} \bybi {211} \bya {120} \bya {221} \bya {200} \bya {002} 
\\
\bya {022} \bya {212} \bybi {121} \byai {011} \byai {001} \byai {100} \bybi {000}
 \end{array}
$}
\qedhere
\end{align*}
\endgroup
\end{proof}

\begin{ack}
This work was partially supported by research grants from the Natural Sciences and Engineering Research Council of Canada.
\end{ack}

\end{document}